\documentclass{article}
\usepackage{amsfonts}
\usepackage{amssymb}
\usepackage{amsmath}
\usepackage{amsthm}
\usepackage{mathrsfs}
\usepackage{amstext}
\usepackage{graphicx}
\usepackage{euscript}

\def\thebibliograph#1#2{\section*{{\normalsize \bf #2}}\list
   {[\arabic{enumi}]}{\settowidth\labelwidth{[#1]}\leftmargin\labelwidth
     \advance\leftmargin\labelsep
     \usecounter{enumi}}
     \def\newblock{\hskip .11em plus .33em minus -.07em}
     \sloppy
     \sfcode`\.=1000\relax}

\newcommand{\re}{{\mathbb{R}}}

\newcommand{\N}{{\mathbb{N}}}
\newcommand{\R}{{\mathbb{R}}^n}

\newcommand{\sobolevh}[2]{\dot{W}^{#1}_{#2}}
\newtheorem{thm}{Theorem}
\newtheorem{prop}{Proposition}
\newtheorem{DE}{Definition}

\newcommand{\be}{\begin{equation}}
\newcommand{\ee}{\end{equation}}
\newcommand{\beq}{\begin{eqnarray}}
\newcommand{\eeq}{\end{eqnarray}}
\newcommand{\beqq}{\begin{eqnarray*}}
\newcommand{\eeqq}{\end{eqnarray*}}

\title{ The higher order chain rule in Sobolev spaces }
  \author{G\'erard Bourdaud}
 \date{\today}

\begin{document}
\maketitle

\begin{abstract}   We establish the Fa\`a di Bruno formula, in the sense of  distributions and almost everywhere, for derivatives of the composed function $f\circ g$, for all function $f:\re\rightarrow \re$ such that $f$ acts on $W^m_p(\R)$ by composition, and all
$g\in W^m_p(\R)$.

\end{abstract}

{\it 2000 Mathematics Subject Classification:} 46E35, 47H30.

{\it Keywords:} Composition Operators. Sobolev spaces.


\section{Introduction}


 If  $g: \R\rightarrow \re$ is sufficiently regular and $\alpha :=(\alpha_1,\ldots,\alpha_n) \in \N^n$, we denote by $g^{(\alpha)}$ the partial derivative
\[ \frac{\partial ^{|\alpha|}g}{\partial x_1^{\alpha_1} \cdots \partial x_n^{\alpha_n}}\,,\]
where $|\alpha|:= \alpha_1+\dots + \alpha_n$. In all the paper $\N$ denotes the set of natural numbers, including $0$.
 If moreover  $f:\re\rightarrow\re$ is sufficiently regular, we set :
\[  U_{\gamma,s}(f,g):=(f^{(s)}\circ g)\, g^{(\gamma_{1})} \ldots g^{(\gamma_{s})}\,, \]
for all integer $s>0$, and all 
\begin{equation}\label{gamma} \gamma:= (\gamma_1,\ldots,\gamma_s)\,,\quad \gamma_r\in\N^n\setminus\{0\}\,,  r=1,\dots ,s\,.\end{equation}

Let $\alpha\in \N^n\setminus\{0\}$. Let us recall the classical  Fa\`a di Bruno's formula :
\begin{equation}\label{faa}
(f\circ g)^{(\alpha)}=\sum c_{\alpha,s, \gamma}U_{\gamma,s}(f,g)\,, \end{equation}
 where the sum runs over all parameters $s, \gamma$ s.t. (\ref{gamma}) and 
 \[s=1,\ldots , |\alpha|\,,\quad  \sum_{r=1}^s \gamma_{r}= \alpha\,, \]  
 and the $c_{\alpha,s,\gamma}$'s are some combinatorial constants.
 This formula holds true as soon as the functions $f:\re\rightarrow\re$ and $g: \R\rightarrow \re$ are differentiable up to order $|\alpha|$.  We are asking ourselves if it remains true for functions $g$ in the Sobolev space $W^m_p(\R)$, if the function $f$ satisfies minimal regularity assumptions.\\
 
 In all the paper  $p$ denotes a real number $\geq 1$. We denote by $c$ a positive constant depending only on $m,n,p$ ; its value may change from place to place. If $E,F$ are function spaces, the notation $E\hookrightarrow F$ means that $E\subseteq F$ and that the natural inclusion mapping $E\rightarrow F$ is continuous. All functions are assumed to be real valued. We use the following function spaces :
 \begin{itemize} 
\item ${\mathcal M}$ the space of all the Lebesgue measurable functions on $\R$ ; ${\mathcal M}_0$ the subspace of all functions $f$ s.t. $f(x)=0$ a.e. ; ${\mathbb M}$ the factor space ${\mathcal M}/ {\mathcal M}_0$. Let $q : {\mathcal M} \rightarrow {\mathbb M}$ be the quotient mapping; if $f\in {\mathcal M}$, then $q(f)$ is called the {\em class} of $f$ ;  if $h\in {\mathbb M}$, any $f\in {\mathcal M}$ s.t. $q(f)=h$ is called a {\em representative} of $h$.
 \item  $L_{1, loc}(\R)\subset {\mathbb M}$ the space of all the classes of locally integrable functions on $\R$.

\item $\mathcal{S}'(\R)$ the space of tempered distributions on $\R$. 
\item $C^m(\R)$ the set of all continuously differentiable functions on $\R$, up to order $m$, and $C^\infty(\R)$ the set of all smooth functions on $\R$.
  \item $C^\ell_b(\re)$ the set of functions on $\re$, with bounded continuous derivatives, up to order $\ell\in \N$.
 \item $C_0(\R)$ the set of continuous functions on $\R$ tending to $0$ at infinity.
 \item For $m\in \N$, $W^m_p(\R)$  (resp.  $\dot{W}^m_p(\R)$) the  inhomogeneous (resp. homogeneous) Sobolev space of distributions $f$ s.t. $f^{(\alpha)}\in L_p(\R)$ for all $|\alpha|\leq m$ (resp. $|\alpha|=m$) endowed with its natural norm (resp. seminorm) :
 \[    \|f\|_{W^{m}_{p}}:= \sum_{|\alpha|\leq m}\|f^{(\alpha)}\|_{p} \quad (\mathrm{resp.} \quad
  \|f\|_{\dot{W}^{m}_{p}}:= \sum_{|\alpha|= m} \|f^{(\alpha)}\|_{p} \, )\,.\]
  \end{itemize}

 Let us consider the composition operator $N_f(g):= f\circ g$. Thirty years ago, we have characterized the functions  $f:\re\rightarrow\re$ s.t. 
 $N_f(W^m_p(\R))\subseteq W^m_p(\R)$, see  \cite{B} and \cite[5.2.4]{RS}. For doing that, we dealt with smooth functions.
By using formula (\ref{faa}) for such functions, we proved an estimate of the following type :
 \begin{equation}\label{apriori}  \| f\circ g\|_{W^m_p} \leq C(f)\, \left(1 + \|g\|_{W^m_p}\right)^m\,,\end{equation}
 where $C(f)$ is the norm of $f$ in a certain function space.
 Then a standard approximation argument (see \cite[5.2.1]{RS}) gives the estimate (\ref{apriori}) for all functions $f,g$
 s.t. $C(f)<\infty$ and $g\in W^m_p(\R)$. As far as we know, the chain rule (\ref{faa}) has not  yet been proved for general functions $f,g$.\\

  We consider here the composition operators acting on the homogeneous  Adams-Frazier space
$\dot{A}^m_p(\R):= (\dot{W}^m_p\cap \dot{W}^1_{mp})(\R)$, endowed with its natural seminorm
\[ \|g\|_{\dot{A}^m_p}:= \|g\|_{\dot{W}^m_p} + \|g\|_{ \dot{W}^1_{mp}}\,.\]
The space $\dot{A}^m_p(\R)$ is well known as 
 a good substitute of $W^m_p(\R)$ when one deals with composition operators, see \cite{AF,Bou_09a}. 
 In particular the Dahlberg degeneracy \cite{Da} does not occur in $\dot{A}^m_p(\R)$.\\
 

\section{Composition operators on $\dot{A}^m_p(\R)$}

\subsection{Normalization of a distribution space}\label{norma}
Let $E$ be a subspace of $L_{1,loc}(\R)$. Assume that any element of $E$ admits a (necessarily unique) continuous representative. Then we {\em may} normalize $E$ as follows :\\

{\bf Convention.} Identify $E$ with the set of all the {\em continuous} functions belonging to $E$.\\

This convention is currently applied to the spaces $W^m_p(\R)$ for $m>n/p$, or $m=n$ and $p=1$.

\subsection{Uniform localization}\label{locunif} 

We consider the space $E_{p}$ of distributions  $f$ on $\re$, which belong to $L_p$ {\em locally uniformly}, i.e.
such that
\[\|f\|_{E_p}:= \left(\sup_{a\in {\mathbb R}} \int_a^{a+1} |f(t)|^p {\rm d}t\right)^{1/p}<+\infty\,.\]
For $\ell\in \N$, let us denote by  $W^{\ell}_{E_p}\hookrightarrow L_{1,loc}(\re)$
the Sobolev space based upon  $ E_p$ -- i.e. the set of distributions $f$ s.t. $f^{(j)}\in E_p$ for all $j=0,\ldots , \ell$ -- endowed with its natural norm, in the same way as $W^m_p$.  
 In case $\ell>0$, $W^{\ell}_{E_p}$ is embedded into $W^1_{p,loc}(\re)$, hence any distribution in $W^{\ell}_{E_p}$ admits a locally absolutely continuous representative. In the sequel, we will use the normalization convention, see section \ref{norma}.
 Hence, in case $\ell>0$, $W^{\ell}_{E_p}$ will be a space of locally absolutely continuous functions.
 The same convention was used implicitly in our former papers about $W^{\ell}_{E_p}$ \cite{Bou_09a,BM}. Restricting to continuous functions does not imply a loss of generality when we deal with composition operators on $\dot{A}^m_p(\R)$. Indeed any function 
 $f:\re\rightarrow \re$, s.t. $N_f$ takes $\dot{A}^m_p(\R)$ to itself, is necessarily continuous, see \cite[thm.~4]{Bou_09a}.
  Let us recall the convenient embedding properties for $W^{\ell}_{E_p}$ :
 
 \begin{prop}\label{dense}  For any $\ell\in \N$, $C^\infty(\re)\cap W^{\ell}_{E_p}$ is a dense subspace of $W^{\ell}_{E_p}$. \end{prop}
\begin{proof} See \cite[prop.~4]{BM}.\end{proof}

 \begin{prop}\label{embed} If $\ell>0$, then $ W^{\ell}_{E_p}$, normalized as above, is embedded into 
$ C^{\ell-1}_b(\re)$\,.\end{prop}
 \begin{proof} See \cite[prop.~5]{BM}.\end{proof}
 

\subsection{Main theorem on composition}\label{compth}


Here is the characterization of composition operators in $\dot{A}^m_p(\re^n)$, see \cite[thm.~1]{Bou_09a} and  \cite[prop.~3 and thm.~2]{BM} :

 \begin{thm}\label{comp} Let $m$ be a natural number $\geq 2$, and $f:\re \rightarrow \re$. \begin{itemize} 
  \item If  $f'\in  W^{m-1}_{E_p}$, then the operator $N_f$ sends $\dot{A}^m_p(\re^n)$ to itself and the estimate
 \be\label{esticompo}
 \|f\circ g\|_{\dot{A}^{m}_p} \leq c\,  \|f'\|_{ W^{m-1}_{E_p}} \left(1+ \|g\|_{\dot{A}^{m}_{p}} \right)^m \ee
holds true for all $g\in \dot{A}^m_p(\re^n)$.   
\item Assume moreover that $\dot{A}^m_p(\re^n)$ is not embedded into $L_\infty(\R)$. If $N_f$ sends $\dot{A}^m_p(\re^n)$ to itself, then $f'\in  W^{m-1}_{E_p}$.
\end{itemize}
\end{thm}

Recall that
$\dot{A}^m_p(\re^n)$ is not embedded into $L_\infty(\R)$, iff
\be\label{crit} m\not= n\quad \mathrm{ or}\quad  p>1\,,\ee see \cite[prop.~3]{BM}.
Proposition \ref{dense} has an useful corollary :

\begin{prop}\label{density}   Let $m$ be a natural number $\geq 1$. Let $f$ be s.t. $f'\in  W^{m-1}_{E_p}$. Then there exists a sequence $(f_k)$ of smooth functions  s.t. 
 \[ \lim_{k\rightarrow \infty} \| f'_k-f'\|_{W^{m-1}_{E_p}}=0\]
 and $f_k \rightarrow f$ uniformly on each bounded subset of $\re$.
 \end{prop}
 
\begin{proof} According to Proposition \ref{dense}, there exists a sequence $(u_k)$ of smooth functions s.t. 
 \[\varepsilon_k:= \| u'_k-f'\|_{W^{m-1}_{E_p}}\] tends to $0$. By Proposition \ref{embed}, we have
$ \| u'_k-f'\|_\infty \leq c\varepsilon_k$. If we define $f_k(x):= u_k(x)-u_k(0)+ f(0)$, we obtain
$| f_k(x)-f(x)| \leq c|x|\varepsilon_k$, for all $x\in \re$.\end{proof}


\section{Main result}

First of all, we need an unambiguous definition of
$U_{\gamma,s}(f,g)\in {\mathbb M}$ if $f'\in  W^{m-1}_{E_p}$ and $g\in \dot{A}^m_p(\R)$.

 Let $u:\re\rightarrow\re$ be any measurable function. If $v, w_r\in {\mathcal M}$ are respective representatives of $g$ and $g^{(\gamma_r)}$, $r=1,\ldots s$, we define 
\[\mathcal{U}_{\gamma,s}(u,v,w_1,\ldots,w_s)(x): = u(v(x))\,w_1(x)\cdots w_s(x)\,. \]
Clearly, the class of  $\mathcal{U}_{\gamma,s}(u,v,w_1,\ldots,w_s)$ does not depend on the choice of the specific representatives $v,w_r$, $r=1,\ldots s$ ; it depends only on $u$ and $g$ ; we denote it by $\mathbb{U}_{\gamma,s}(u,g)$.

 In case $s<m$, by the normalization convention and by Proposition \ref{embed}, the function $f^{(s)}$ is continuous. There is no need to choose another representative than $f^{(s)}$ itself. 
 Then we set  $U_{\gamma,s}(f,g):= \mathbb{U}_{\gamma,s}(f^{(s)},g)$. That works as well for $s=m$ if $f\in C^m(\R)$.
 
 \begin{DE}\label{welldef} If
$\mathbb{U}_{\gamma,m}(u_1,g) =\mathbb{U}_{\gamma,m}(u_2,g)$ for any representatives $u_1,u_2$ of $f^{(m)}$,
we say that {\em  $U_{\gamma,m}(f,g)$ is well defined}, and we set
$U_{\gamma,m}(f,g):= \mathbb{U}_{\gamma,m}(u,g)$  for any representative $u$ of $ f^{(m)}$.
\end{DE}


  Our main theorem is the following :

 \begin{thm}\label{compU} Let $m$ be a natural number $\geq 2$. Let $f:\re \rightarrow \re$ be a function s.t.  $f'\in  W^{m-1}_{E_p}
 $. 
 Then, for all function $g\in \dot{A}^m_p(\R)$,   \begin{itemize}
 \item
$U_{\gamma,s}(f,g)$ is well defined  and satisfies
 \begin{equation}\label{esticomp} \|U_{\gamma,s}(f,g)\|_{pm/\ell}
 \leq c\,  \|f'\|_{ W^{m-1}_{E_p}} \left(1+ \|g\|_{\dot{A}^{m}_{p}} \right)^m ,\end{equation}
 for all $s=1,\ldots ,m$, $\gamma$ s.t. (\ref{gamma}) and \be\label{goodh} \ell:= \sum_{r=1}^s |\gamma_r| \leq m\,.\ee
  
  \item The higher order chain rule (\ref{faa}) holds true in the sense of distributions, and almost everywhere, for all $\alpha\in \N^n$ s.t.  $0<|\alpha| \leq m$.
  \end{itemize}
 \end{thm}

 \subsection{Proof of the estimate (\ref{esticomp}) in case $s<m$}\label{s<m}
 
 This part of the proof is classical. For the sake of completeness, we outline the argument.

 We exploit the Gagliardo-Nirenberg embedding :
  \be\label{GN} 
  \dot{A}^m_p(\re^n) \hookrightarrow   \dot{W}^\ell_{pm/\ell}(\re^n)\,,\quad \ell=1,\ldots , m\,,
  \ee
 together with the estimate
 \be\label{interpol2}
 \|g\|_{\dot{W}^{\ell}_{pm/\ell}} \le c\,\|g\|_{\dot{W}^{1}_{pm}}^{\frac{m-\ell}{m-1}}\,
\|g\|_{\dot{W}^{m}_{p}}^{\frac{\ell-1}{m-1}}\,.\ee
 According to Proposition \ref{embed}, $f^{(s)}\in C^0_b(\re)$ and
 $\| f^{(s)} \|_\infty \leq c \, \|f'\|_{ W^{m-1}_{E_p}}$. Then $f^{(s)}\circ g$ is a bounded measurable function s.t.
 $\|f^{(s)}\circ g\|_\infty\leq c \|f'\|_{ W^{m-1}_{E_p}}$.
 By H\"older inequality, and by (\ref{goodh}), we obtain
\[ \|g^{(\gamma_{1})}
\cdots
 g^{(\gamma_{s})}\|_{pm/\ell}\le
 \prod_{r=1}^{s}\|g\|_{\sobolevh{|\gamma_{r}|}{pm /|\gamma_{r}|}}\,. \]
 By (\ref{interpol2}), 
the right hand side of the above inequality is estimated by
\[
\left( \prod_{r=1}^{s}\|g\|_{\sobolevh{1}{pm}}^{\frac{m-|\gamma_{r}|}{m-1}}\right)
 \left( \prod_{r=1}^{s}\|g\|_{\sobolevh{m}{p}}^{\frac{|\gamma_{r}|-1}{m-1}}\right)\,.
\]
 By summing up the exponents, we obtain
\[\| g^{(\gamma_{1})}
\cdots
 g^{(\gamma_{s})}\|_{pm/\ell}
\le c\,\|g\|_{\sobolevh{1}{pm}}^{\frac{sm-\ell}{m-1}}
\|g\|_{\sobolevh{m}{p}}^{\frac{\ell-s}{m-1}} \leq c\, \|g\|^s_{\dot{A}^m_p}
\,.\]
That implies the estimate (\ref{esticomp}).

\subsection{Proof of the estimate (\ref{esticomp}), and the well-definiteness of $U_{\gamma,m}$}\label{s=m}
In that case, there exist  integers
$j_r\in \{1,\ldots,n\}$, $r=1,\ldots , m$, s.t. 
 $U_{\gamma,m}(f,g) = (f^{(m)}\circ g)\, \partial_{j_1}g\cdots \partial_{j_m}g$ for smooth functions $f,g$.
 
 \subsubsection{Distinguished representatives} Let us define
  $A(\R)$ as the set of measurable functions on $\R$, which are locally absolutely continuous on almost every line parallel to any coordinate axis.
     The pertinency of the class $A(\R)$ in the framework of Sobolev spaces relies upon the following classical result, due to Gagliardo, Morrey and Calkin (see \cite{G}, \cite[Lem.~1.5]{MM}, \cite[1.1.3, thm.~1]{MA}): 
     
  \begin{prop}\label{GMC} All $g\in W^1_p(\R)$, admit a representative $v\in A(\R)$ s.t. the usual gradient of $v$ is a representative of the distributional gradient of $g$.      \end{prop}
  
  In the remaining of Section \ref{s=m}, we consider a fixed $g\in \dot{A}^m_p(\R)$. Since $g\in \dot{A}^1_{mp}(\R)$, then $g$ belongs locally to  $W^1_{mp}(\R)$. We fix a representative $v$ of $g$ satisfying the conditions of
Proposition \ref{GMC}, that we call
a {\em distinguished representative}.
       
 \subsubsection{The case of a smooth function $f$}\label{casesmooth}

We follow the proof of \cite[Lemma 1]{BM}, with 
  the same notation. Without loss of generality, we replace $U_{\gamma,m}(f,g)$ by 
\[ S_{D}(g) :=(f^{(m)}\circ g)\, (Dg)^m\,,\]
where $D$ is any first order differential operator with constant coefficients. 
Then 
  \be\label{goodIP}
  \|S_{D}(g)\|_p^p= \sum_{j=1,2} \int_{\R} (F_j\circ g)  \, T_j(g)\,{\rm d}x\,,
  \ee
  where $F_1,F_2$ are bounded continuous functions  on $\re$ s.t. $\|F_j\|_\infty \leq c\,  \|f'\|^p_{ W^{m-1}_{E_p}}$, and
 \[T_1(g):=  (mp-1) \,(D^2g)\, |Dg|^{mp-2}\,,\quad
T_2(g):=  |Dg|^{mp} \,.\]
Notice that the proof of (\ref{goodIP}) makes use of the distinguished representative of $g$.
Then $\|T_2(g)\|_1 \leq c\, \| g\|^{mp}_{\sobolevh{1}{mp}}$
 and, by H\"older and Gagliardo-Nirenberg,
  \[\|T_1(g)\|_1 \leq c \,  \|g\|_{\sobolevh{1}{pm}}^{mp-2}\,
\|g\|_{\sobolevh{2}{pm/2}} \leq c_1\, \|g\|_{\sobolevh{1}{pm}}^{mp - \frac{m}{m-1}}\,
\|g\|_{\sobolevh{m}{p}}^{\frac{1}{m-1}}\leq c_1 \|g\|^{mp-1}_{\dot{A}^{m}_{p}}\,.\]
We obtain 
\[ \|S_D(g)\|_p\leq c\,  \|f'\|_{ W^{m-1}_{E_p}}\,\left( \|g\|^{m-(1/p)}_{\dot{A}^{m}_{p}} +  \|g\|^{m}_{\dot{A}^{m}_{p}}\right)\,.\]
That implies the estimate (\ref{esticomp}).\\

\subsubsection{The general case} By Section \ref{casesmooth},
the operator $f\mapsto U_{\gamma,m}(f,g)$ can be uniquely extended by density to the whole set of functions $f$ s.t. $f'\in W^{m-1}_{E_p}$. Let us define $V_{\gamma,m}(f,g)\in L_p(\R)$ as the limit of the sequence
$( U_{\gamma,m}(f_k,g))$, for any
sequence $(f_k)$ of smooth functions approximating $f$ in the sense of Proposition \ref{density}. It holds
\begin{equation}\label{opV}\| V_{\gamma,m}(f,g)\|_p \leq c\,  \| f'\|_{W^{m-1}_{E_p}} (1 + \|g\|_{\dot{A}^{m}_{p}})^m\,.\end{equation}

 Let us consider any representative $u$ of $f^{(m)}$. Let $h$ be a representative of $V_{\gamma,m}(f,g)$.
 We claim that
\begin{equation}\label{concrete} h(x) =  \mathcal{U}_{\gamma,m}(u,v,\partial_{j_1}v,\ldots,\partial_{j_m}v)(x)\quad \mathrm{ a.e.}\,.\end{equation}
For convenience, we assume that $j_1=1$. By assumption on $v$, there exists a set $N$ of measure $0$ in $\re^{n-1}$ s.t., for every $x'\notin N$, the function $t\mapsto v(t,x')$ is locally absolutely continuous on $\re$. Let $(f_k)$ be a sequence of smooth functions  approximating $f$ in the sense of Proposition \ref{density}. By taking a subsequence, if necessary, we have
\[f_k^{(m)}(t) \rightarrow u(t) \quad\mathrm{for\,almost\,every}\, t\in \re\,,\]
\begin{equation}\label{convae2}   \mathcal{U}_{\gamma,m}(f_k^{(m)},v,\partial_{j_1}v,\ldots,\partial_{j_m}v)(x)\rightarrow  h(x) \quad\mathrm{for\,almost\,every}\, x\in \R\,.\end{equation}

Let $M$ be the set of $t\in \re$ s.t. $f_k^{(m)}(t)$ does not converges to $u(t)$. Then $M$ is a set of measure $0$. Applying de la Vall\'ee-Poussin's lemma (\cite{VP} and \cite[lem.~1.1]{MM}), we deduce that the set
\[B_{x'}:= \{\, t\in \re\,|\, v(t,x')\in M \quad \mathrm{and}\quad \partial_1v(t,x')\not=0\,\}\]
is  of measure $0$, for any $x'\notin N$. Let us define :
\[C:= \{(t,x')\,|\, t\notin B_{x'} \, \mathrm{and}\, x'\notin N\}\,.\]
Then $\R\setminus C$ is a set of measure $0$ and $f_k^{(m)}(v(t,x')))\,\partial_1v(t,x') \rightarrow u(v(t,x'))\, \partial_1v(t,x')$ for all $(t,x')\in C$. By using the property (\ref{convae2}), we  may conclude that the property (\ref{concrete}) is satisfied. We conclude that
\[ V_{\gamma,m}(f,g)= {\mathbb U}_{\gamma,m}(u,g)\,.\]
It follows that 
${\mathbb U}_{\gamma,m}(u,g)$ does not depend on the specific representative  $u$. 
 Hence $U_{\gamma,m}(f,g)$ is well defined, and $U_{\gamma,m}(f,g)=  V_{\gamma,m}(f,g)$. The estimate (\ref{esticomp}) follows by (\ref{opV}).

\subsection{Proof of the chain rule}\label{proofchain}

\subsubsection{A preparation}\label{real}

The function space $ \dot{A}^m_p(\re^n)$ is endowed with a seminorm s.t. $\|g\| =0$ iff $g$ is a constant function.
We have the following statement, of immediate proof :
  \begin{prop}\label{constant} Let  $g$ be  a function in $\dot{A}^m_p(\re^n)$ satisfying the following property : the chain rule (\ref{faa}) holds true for all $f'\in  W^{m-1}_{E_p}$ and all $0<|\alpha|\leq m$. Then, for all number $a\in \re$, the function $g+a$ satisfies the same property.
  \end{prop} 
  
  Let $E$ be a subspace of $ \dot{A}^m_p(\re^n)$ s.t. 
$ \dot{A}^m_p(\R)= E \oplus \mathcal{P}_0$, where  $\mathcal{P}_0$ denotes the set of constant functions. Consider the  associated projection mapping $R :  \dot{A}^m_p(\R) \rightarrow \mathcal{S'}(\R)$ s.t. $R(g)$ is the unique element of $E$ 
s.t. $g-R(g)\in  \mathcal{P}_0$. If $R$ is a continuous mapping from  $\dot{A}^m_p(\R)$, endowed with its semi-norm,
to $\mathcal{S'}(\R)$, endowed with the $\ast$-weak topology, $E$ is called a {\em realization} of $\dot{A}^m_p(\R)$ ; if endowed with the {\em norm}
$\|.\|_{ \dot{A}^m_p}$, $E$  turns out to be a Banach space, continuously embedded into $L_{p, loc}(\R)$.  We refer to \cite[4.1]{BM} for details. Till the end of Section \ref{proofchain}, $E$ will denote a realization of $\dot{A}^m_p(\R)$.

\begin{prop}\label{cont} Under the conditions of Theorem \ref{compU}, the mapping
\ $g \mapsto U_{\gamma,s}(f,g)$
is continuous from $E$ to $L_{pm/\ell}$.
\end{prop}

\begin{proof} {\em Step 1 : case $s<m$.}  As observed in Section \ref{s<m}, the function $f^{(s)}$ is bounded and continuous, and the mapping $g \mapsto  g^{(\gamma_1)}\cdots g^{(\gamma_s)}$ is continuous from $E$ to $L_{pm/\ell}$.
Then the wished property follows by \cite[prop.~9]{BM}.

{\em Step 2 : case $s=m$.} The continuity of $g \mapsto U_{\gamma,m}(f,g)$ is proved in \cite[5.1.2]{BM}. \end{proof}

\subsubsection{The case of a smooth function $f$}\label{chainsmooth}

Let $g\in E$. Let us consider a sequence $(g_k)$ a smooth functions s.t. $g_k\rightarrow g$ in $E$. Then
 \[
 (f\circ g_k)^{(\alpha)}=\sum c_{\alpha,s, \gamma}U_{\gamma,s}(f,g_k) \]
 holds true in the usual sense. A standard argument shows that  
 \[(f\circ g_k)^{(\alpha)} \rightarrow (f\circ g)^{(\alpha)}\]
 in the sense of distributions, see e.g. \cite[step 2, p.~6111]{Bou_09a}. By Proposition \ref{cont},
 $U_{\gamma,s}(f,g_k)$ tends to $U_{\gamma,s}(f,g)$ in $L_{pm/|\alpha|}$, hence in the sense of distributions. 
 Then the chain rule holds true in the sense of distributions.  By Theorem \ref{comp}, and embedding (\ref{GN}), we know that  $(f\circ g)^{(\alpha)}$ belongs to $L_{pm/|\alpha|}(\R)$. We conclude that the chain rule holds true a.e..

  \subsubsection{The general case} \label{chain}

  Let $g\in E$. We use again a sequence $(f_k)$ given by Proposition \ref{density}. By Section \ref{chainsmooth}, we have
  \[
 (f_k\circ g)^{(\alpha)}=\sum c_{\alpha,s, \gamma}U_{\gamma,s}(f_k,g)\,, \]
 a.e., for any $k$.  By the estimate (\ref{esticomp}), $U_{\gamma,s}(f_k,g)$ tends to $U_{\gamma,s}(f,g)$ in $L_{pm/|\alpha|}$, hence in the sense of distributions.  By Proposition \ref{density}, it follows easily that
  $f_k\circ g \rightarrow f\circ g$ in $L_{1,loc}(\R)$. Hence
 \[(f_k\circ g)^{(\alpha)} \rightarrow (f\circ g)^{(\alpha)}\]
 in the sense of distributions.  We conclude the proof so as in Section \ref{chainsmooth}.
 
 \section{Related results}
 
 Now we state two corollaries of Theorem \ref{compU}.
 
\begin{thm}\label{chainbis}  Let $m$ be a natural number $\geq 2$. Let $f:\re \rightarrow \re$ be a function s.t.  $N_f$ sends $\dot{A}^m_p(\re^n)$ to itself. Let $\alpha\in \N^n$ be s.t. $0<|\alpha|\leq m$.  Then the formula (\ref{faa}) holds true in the sense of distributions and almost everywhere, for all $g\in \dot{A}^m_p(\re^n)$.
 \end{thm}
 
\begin{proof}If the condition (\ref{crit}) holds, Theorem \ref{comp} says us that $f'\in W^{m-1}_{E_p}$ ; then the chain rule follows by Theorem \ref{compU}.
 
 Assume now that $N_f$ sends $\dot{A}^n_1(\re^n)$ to itself. We know that $\dot{A}^n_1(\re^n)$ is a subspace of $C_b(\R)$,  see \cite[prop.~3]{BM}. According to \cite[thm.~3]{BM}, it holds
$ f\in W^n_1(\re)_{loc}$. Let $g\in \dot{A}^n_1(\re^n)$. We may consider a smooth compactly supported function  $\psi$ s.t. $\psi(x)=1$ on the range of $g$. Then $f\circ g= (f\psi)\circ g$ and $(f\psi)'\in W^{n-1}_{E_1}$. By applying Theorem \ref{compU} to the function
$f\psi$, we obtain the formula (\ref{faa}).\end{proof}

 Theorem \ref{chainbis} has a counterpart for usual inhomogeneous Sobolev spaces. Here we consider the inhomogeneous
 Adams-Frazier space $A^m_p(\R):= (W^m_p\cap \dot{W}^1_{mp})(\R)$, endowed with its natural norm.

\begin{thm}\label{chainter}  Let $m$ be a natural number $\geq 2$. Let $f:\re \rightarrow \re$ be a function s.t.  $N_f$ sends $A^m_p(\re^n)$ (resp. $W^m_p(\R)$) to itself. Let $\alpha\in \N^n$ be s.t. $0<|\alpha|\leq m$.  Then the formula (\ref{faa}) holds true in the sense of distributions and almost everywhere, for all $g \in A^m_p(\re^n)$ (resp. $W^m_p(\R)$).\end{thm}

\begin{proof} {\em Step 1 : the case of $A^m_p(\re^n)$. }
 
 {\em Substep 1.1.} If
 $m>n/p$, or $m=n$ and $ p=1$, then
 $A^m_p(\re^n)=W^m_p(\R)\hookrightarrow C_0(\R)$.
 According to \cite[thm.~3]{BM}, it holds $ f\in W^m_p(\re)_{loc}$. Then we conclude the proof, exactly as that of Theorem \ref{chainbis}.
 
 {\em Substep 1.2.} If
 $m\leq n/p$, and $m\not= n$ or $ p>1$, then we have $f'\in W^{m-1}_{E_p}$,
 see \cite[thm.~2]{BM}, and we apply Theorem \ref{compU}.\\
 
 {\em Step 2 : the case of $W^m_p(\re^n)$. }
 
  {\em Substep 2.1.} If $m\geq n/p$, then $W^m_p(\re^n)= A^m_p(\re^n)$, and we are reduced to Step 1. 
  
  {\em Substep 2.2.} If
 \[ 1 + \frac{1}{p} < m< \frac{n}{p}\,,\]
 it holds $f(t)= ct$ for some constant $c$, see \cite{Da}, then the chain rule holds trivially.
 
 {\em Substep 2.3.} If $n=1,2$, then $W^2_1(\R)= A^2_1(\R)$ and we are reduced to Step 1.  

 {\em Substep 2.4.} Let us turn to the case $m=2, p=1, n\geq 3$.  According to \cite[thm.~3]{B}, it holds $f''\in L_1(\re)$.
 Arguing so as in Section \ref{s=m}, we have 
\be\label{casL1} \| (f''\circ g)\, (\partial_jg) (\partial_kg)\|_1 \leq  \|f''\|_1 \| g\|_{W^2_1}\ee
 for all $g\in W^2_1(\R)$, $j,k = 1,\ldots, n$, and for any smooth function $f$.
 Then, arguing so as in Section \ref{proofchain}, we obtain the chain rule of order 2 for functions
 $g\in W^2_1(\R)$ and for any $f$ s.t. $f''\in L_1(\re)$. The details are left to the reader.\end{proof}

{

		G\'erard Bourdaud
	
	Universit\'e Paris Cit\'e and Sorbonne Universit\'e, 
	
	CNRS, IMJ-PRG, 
	
	F-75013 Paris, France
	
		bourdaud@math.univ-paris-diderot.fr  \\

 \end{document}